\documentclass[12pt,reqno]{amsart}
\usepackage{amsfonts}
\usepackage{amssymb}
\usepackage{a4wide}
\usepackage{pgf,pgfarrows,pgfautomata,pgfheaps,pgfnodes,pgfshade}
\usepackage{url}
\usepackage{lineno}
\usepackage{graphicx}
\usepackage[all]{xy}
\usepackage{relsize}
\usepackage{hyperref}

\numberwithin{equation}{section}

\newtheorem{defi}{Definition}[section]
\newtheorem{thm}[defi]{Theorem}
\newtheorem{lemm}[defi]{Lemma}
\newtheorem{rem}[defi]{Remark}

\newtheorem{cor}[defi]{Corollary}
\newtheorem{prop}[defi]{Proposition}
\newtheorem{exam}[defi]{Example}

\newcommand{\esssup}{{\mathrm{ess}\sup}}

\newcommand{\R}{\mathbb{R}}

\begin{document}

\title[Malliavin and flow regularity of SDEs.]
{Malliavin and flow regularity of SDEs. Application to the study of densities and the stochastic transport equation.}
\author[D. Ba\~{n}os]{David Ba\~{n}os}
\address{D. Ba\~{n}os: CMA, Department of Mathematics, University of Oslo, Moltke Moes vei 35, P.O. Box 1053 Blindern, 0316 Oslo, Norway.}
\email{davidru@math.uio.no}
\author[T. Nilssen]{Torstein Nilssen}
\address{T. Nilssen: CMA, Department of Mathematics, University of Oslo, Moltke Moes vei 35, P.O. Box 1053 Blindern, 0316 Oslo, Norway.}
\email{t.k.nilssen@cma.uio.no}

\maketitle


\begin{abstract}
In this work we present a condition for the regularity, in both space and Malliavin sense, of strong solutions to SDEs driven by Brownian motion. We conjecture that this condition is optimal. As a consequence, we are able to improve the regularity of densities of such solutions.

We also apply these results to construct a classical solution to the stochastic transport equation when the drift is Lipschitz.
\end{abstract}
 


\section{Introduction}


This paper is mainly divided into two parts. First, we are interested in studying the regularity properties of the following Stochastic Differential Equation (SDE)
\begin{align}\label{SDE0}
dX_t = b(t,X_t)dt+dB_t, \ 0\leq t \leq T, \ X_0 =x\in \R^d,
\end{align}
where $B_t$, $t\in [0,T]$ is a $d$-dimensional Brownian motion and $b$ is a measurable function such that a unique strong solution exists. Our goal is to analyse the regularity of strong solutions to (\ref{SDE0}) both in space and in the Malliavin sense. We give a condition based on the regularity properties of $b$ to obtain regularity properties of $X_t$, $t\in [0,T]$. Then we study the consequences of the aforementioned properties and we take two different directions. On one hand, the Malliavin regularity allows us to improve the regularity of densities of strong solutions. On the other hand, the regularity in space entitles us to study the associated Stochastic Transport Equation and gain more regularity on the solution. Namely, for $b$ Lipschitz we are able to show that one obtains a classical solution to the Stochastic Transport Equation.

Considerable research in the direction of regularity of densities of solutions to SDEs has been done in the past years. There are well-known results on conditions for a density to be smooth when the coefficients are smooth, for example, in \cite{Nua10} or in the case of SDEs with boundary conditions in \cite{KHSS97} under the so-called H\"{o}rmander's condition.

We highlight the work by S. Kusuoka and D. Stroock in \cite{Kus82} where the authors show that if \mbox{$b\in C_b^{n+2}(\R^d)$} then the density lies in $C_b^n(\R^d)$ using Sobolev inequalities associated to the \mbox{$H$-derivative} of the solution. Here, we improve the regularity of the density and skip the boundedness of $b$, instead we consider additive noise and provide an extension to a class of non-degenerate diffusion coefficients. In \cite{Kus10} S. Kusuoka also gives conditions for the law to be absolutely continuous with respect to the Lebesgue measure when drift coefficients are non-Lipschitz, his work is closely related to the findings of N. Bouleau and F. Hirsch in \cite{BH91} where the authors show that for global Lipschitz coefficient a density exists, here we show that such density is H\"{o}lder continuous with exponent $\alpha<1$. An improvement was given in dimension one in \cite{FP10} where they show that for a H\"{o}lder continuous drift of, at most, linear growth and H\"{o}lder continuous diffusion coefficient the solution admits densities at any given time.

Our technique is mainly based on Malliavin calculus and on a sharp estimate on the moments of the derivative of the flow associated to the solution motived by a previous work in \cite{MNP14} where a very similar estimate is used to prove that solutions of SDE's with irregular drift coefficients are once Sobolev differentiable in the initial condition. SDE's with irregular drift coefficients has been a very active topic of research in the last years. For example, it is known, see \cite{Men11}, where the authors prove that solutions of SDEs with merely bounded drift are Malliavin differentiable. Nevertheless, once Malliavin differentiability is not enough to guarantee regularity of the densities. A related result on this matter is \cite{AP14} where the drift is assumed to be irregular and a representation for the derivative is also obtained, even in the higher dimensional case. Also, \cite{Zhang11} deals with the same problem as in \cite{MNP14} but here, the diffusion is non-trivial. Here instead, we look at the regularity of the drift coefficient and connect it to regularity of the solution in both the Malliavin and Sobolev sense.

For the study of densities, we use a powerful result by V. Bally and L. Caramellino in \cite{Bally2011} which allows us directly to improve the regularity of densities of solutions with with sufficient Malliavin regularity. In addition, we also look at the regularity in space. As a consequence of the relationship between the Malliavin and Sobolev derivatives we are also able to give a condition to determine the regularity  of solutions to (\ref{SDE0}) in the Sobolev sense (locally) and show that such derivatives admit moments of any order. At the end of the section we also give an extension to more general diffusions. 

The last application of the paper is devoted to the study of the Stochastic Transport Equation (STE) since it is closely related to the SDE (\ref{SDE0}) by the inverse of the flow of the solution. We use the results obtained in the first part of the paper to show that, for $b$ Lipschitz, the solution is classical. Work in the direction of SPDE's and in particular the Stochastic Transport Equation has brought a lot of interests in the last years. In \cite{FGP10} F. Flandoli, M. Gubinelli and E. Priola study the well-posedness for H\"{o}lder-continuous drifts and show pathwise uniqueness of the weak solution. In \cite{Nil13}, in dimension one, it is shown that when the drift is a step function then the solution to the transport equation is even once continuously differentiable.

\section{Framework}

In this section we recall some facts from Malliavin calculus and Sobolev spaces, which we aim at employing in Section \ref{main results} to analyse the regularity of densities of strong solutions of SDEs. See \cite{Nua10, Mall78, Mall97, DOP08} for a deeper insight on Malliavin Calculus. As for theory on Sobolev spaces the reader is referred to \cite{Leo09, Evan10}.

\subsection{Basic elements of Malliavin Calculus}\label{somedefmalcal}

In this Section we briefly elaborate a framework for Malliavin calculus.

Let $H$ be a real separable Hilbert space and $W=\{W(h), h\in H\}$ an isonormal Gaussian process, see \cite[Definition 1.1.1]{Nua10}. Assume $W$ is defined on a complete probability space $\left( \Omega ,\mathcal{F},P\right)$ and $\mathcal{F}$ is generated by $W$.

Denote by $D$ the derivative
operator acting on elementary smooth random variables in the sense that%
\begin{equation*}
D(f(h_{1},\ldots,h_{n}))=\sum_{i=1}^{n}\partial
_{i}f(h_{1},\ldots,h_{n})h_{i},\quad h_{i}\in H, \quad f\in C_{b}^{\infty }(\mathbb{R%
}^{n}).
\end{equation*}
Further let $\mathbb{D}^{k,p} (\Omega)$, $k,p\geq 1$ be the closure of the family of elementary smooth random variables with respect to the norm%
\begin{equation*}
\left\Vert F\right\Vert _{\mathbb{D}^{k,p}(\Omega)}:=\left\Vert F\right\Vert _{L^{p}(\Omega
)}+\sum_{i=1}^k \left\Vert D\overset{i)}{\cdots} DF\right\Vert _{L^{p}(\Omega ;H\otimes \overset{i)}{\cdots} \otimes H)}.
\end{equation*}

Our framework will rely on the special case where $H$ is isometric to $L^2([0,T];\R^d)$ endowed with the Lebesgue measure, then we have that the Malliavin derivative is a process $\{D_t F\}_{t\in [0,T]}$ in $L^2(\Omega\times [0,T];\R^d)$ defined as
$$D_tF = \sum_{i=1}^n \frac{\partial}{\partial x_i} f\left(\int_0^T h_1(u)dW_u, \dots, \int_0^T h_n(u)dW_u \right) h_i(t)$$
and again we take the closure w.r.t. the norm
$$\|F\|_{\mathbb{D}^{k,p}(\Omega)} = E[|F|^p]^{1/p}+ \sum_{i=1}^k E\left[ \int_0^T\cdots \int_0^T \|D_{t_1}\cdots D_{t_i} F\|^p dt_1\cdots t_i\right]^{1/p}$$
where $\|\cdot\|$ denotes any norm in $\R^{d\times \overset{i)}{\cdots} \times d}$.

The operator $D\overset{k)}{\cdots} D$ is then a closed operator from $\mathbb{D}^{k,p} (\Omega)$ to $L^p(\Omega \times [0,T]^k ; \R^{d\times \overset{k)}{\cdots} \times d})$ for all $p\geq 1$. Moreover, for $p\leq q$ and $k\leq l$ we have
$$\|F\|_{\mathbb{D}^{k,p}(\Omega)} \leq \|F\|_{\mathbb{D}^{l,q}(\Omega)}$$
and as a consequence
$$\mathbb{D}^{k+1,p}(\Omega) \hookrightarrow \mathbb{D}^{k,q} (\Omega)$$
if $k\geq 0$ and $p>q$.

We shall say that a random variable is $k$-times \emph{Malliavin differentiable} with derivatives in $L^p(\Omega)$, $p\geq 1$ if it lies on $\mathbb{D}^{k,p}(\Omega)$.

Finally, we have the chain-rule for the Malliavin derivative. Let $\varphi :\R^m\rightarrow \R^m$ be a function such that
$$|\varphi(x) - \varphi(y)| \leq K |x-y|$$
for any $x,y \in \R^m$. Suppose $F=(F^1,\dots,F^m)$ is a random vector whose components belong to the space $\mathbb{D}^{1,2}(\Omega)$. Then $\varphi(F) \in \mathbb{D}^{1,2}(\Omega)$ and there exists a random vector $G=(G_1,\dots,G_m)$ bounded by $K$ such that
$$D \varphi(F) = \sum_{i=1}^mG_i DF^i.$$

In particular if $\varphi'$ exists, then $G= \varphi'(F)$.

For a stochastic process $v \in \mbox{Dom}( \delta)$ (not necessarily adapted to $\{\mathcal{F}_t\}_{t\in [0,T]}$) we denote by
\begin{align}\label{skorokhod}
\delta(v) := \int_0^T v_t \delta B_t
\end{align}
the action of $\delta$ on $v$. The above expression (\ref{skorokhod}) is known as the Skorokhod integral of $v$ and it is an anticipative stochastic integral. It turns out that all $\{\mathcal{F}_t\}_{t\in [0,T]}$-adapted processes in $L^2(\Omega\times[0,T])$ are in the domain of $\delta$ and for such processes $v_t$ we have
$$\delta( v ) =\int_0^T v_t dB_t,$$
i.e.the Skorokhod and It\^{o} integrals coincide. In this sense, the Skorokhod integral can be considered to be an extension of the It\^{o} integral to non-adapted integrands.

The dual relation between the Malliavin derivative and the Skorokhod integral implies the following important formula:

\begin{thm}[Duality formula]\label{duality}
Let $F\in \mathbb{D}^{1,2}$ and $v\in \mbox{Dom}(\delta)$. Then
\begin{align}
E\left[ F\int_0^T v_t \delta B_t\right] = E\left[ \int_0^T v_t D_t F dt\right].
\end{align}
\end{thm}

\subsection{Basic facts of theory on Sobolev spaces}

In this section we concisely review some basic facts about theory on Sobolev spaces.

Let $U$ be an open bounded subset of $\R^d$. Fix $1\leq p \leq \infty$ and let $k\geq 0$ an integer. The Sobolev space $W^{k,p}(U)$ is composed by all locally integrable functions $u: U \rightarrow \R$ such that for any multiindex $\alpha$ with $|\alpha|\leq k$, then $D^\alpha u$ exists in the weak sense and belongs to $L^p(U)$.

We endow the space $W^{k,p}(U)$ with the topology generated by the norm
$$\|u\|_{W^{k,p}(U)} := \left(\sum_{|\alpha|\leq k} \int_U |D^\alpha u|^p dx \right)^{1/p}, \ \ 1\leq p <\infty$$
or
$$\|u\|_{W^{k,\infty}(U)}:= \sum_{|\alpha|\leq k} \esssup_{U} |D^\alpha u|, \ \ p=\infty. $$

The following relations will be of high relevance for our purposes. For $1\leq p < q \leq \infty$, $k>l$ such that $(k-l)p<d$ and
$$\frac{1}{q} = \frac{1}{p}- \frac{k-l}{d}$$
then we have the following continuous embedding
\begin{align}
W^{k,p}(\R^d) \hookrightarrow W^{l,q}(\R^d).
\end{align}

Also, we have the following embedding as a consequence of Morrey's inequality; if $\frac{k-r-\alpha}{d}=\frac{1}{p}$ with $\alpha\in (0,1)$ then
\begin{align}\label{embed}
W^{k,p}(\R^d) \hookrightarrow C^{r,\alpha}(\R^d).
\end{align}
Essentially, this means that if we have enough Sobolev-regularity then we may expect some continuous classical derivatives up to some order.

We will though use $\frac{\partial}{\partial x}$ to denote differentiation in both the weak and classical sense when the context is clear.

\subsection{Shuffles}\label{shuffles}
Let $m,n\in \mathbb{N}_0$ and denote by $S_{m} =\{\sigma: \{1,\dots, m\}\rightarrow \{1,\dots,m\} \}$ the set of permutations of $m$ elements. Define the set of \emph{shuffle permutations} of length $n+m$ as
$$S(m,n) := \{\sigma\in S_{m+n}: \, \sigma(1)<\cdots <\sigma(m), \, \sigma(m+1)<\cdots <\sigma(m+n)\}.$$

Fix $s,t\in [0,T]$ with $s <t$ and define the $m$-dimensional subset of $[0,1]^m$
$$\Lambda_{s,t}^m:=\{(u_1,\dots,u_m)\in [0,T]^m : \, s<u_1<\cdots < u_m<t\}.$$

Denote $u=(u_1,\dots,u_m)\in [0,T]^m$ and for given a permutation $\sigma \in S_m$ let $u_{\sigma(1:m)} := (u_{\sigma(1)},\dots, u_{\sigma(m)})$.

Let $f:[0,T]^m \rightarrow \R$ and $g:[0,T]^n\rightarrow \R$ be two measurable functions. Then
\begin{align}\label{shuffle}
\left(\int_{\Lambda_{s,t}^m} f(u)du \right)\left( \int_{\Lambda_{s,t}^n} g(u)du \right)= \sum_{\sigma\in S(m,n)} \int_{\Lambda_{s,t}^{m+n}} f(u_{\sigma(1:m)})g(u_{\sigma(m+1:m+n)})du.
\end{align}

We will also need the following formula. Let $f_i:[0,T] \rightarrow \R$, $i=1,\dots,m+n$ be measurable functions and $k=0,\dots, m$ be fixed. Then 
\begin{align}\label{shuffle2}
\begin{split}
\int_{\Lambda_{s,t}^m} \int_{\Lambda_{s,u_k}^n} \prod_{i=1}^n f_i(v_i) \prod_{i=n+1}^{n+m} f_i(u_i) dvdu = \sum_{\sigma\in S_k(m,n)} \int_{\Lambda_{s,t}^{m+n}} \prod_{i=1}^{m+n} f_i(w_{\sigma(i)}) dw
\end{split}
\end{align}
where
$$S_k (m,n) = \{\sigma \in S(m,n): \, \sigma(j)=j, \, j=k,\dots,m\}.$$
If we consider the case $u_k=t$ in the last formula we indeed obtain \eqref{shuffle}. Observe also that the number of terms in the sum given in \eqref{shuffle2} is at most of order $C^{m+n}$ for some finite constant $C>0$.


\section{Malliavin and flow regularity of strong solutions of SDEs}\label{main results}

Consider the \emph{stochastic differential equation} (SDE) given by
\begin{align}\label{SDE}
\begin{cases}
dX_t = b(t,X_t) dt + dB_t,\\
X_0=x \in \R^d,
\end{cases}
\end{align}
where the drift coefficient $b:[0,T]\times \R^d \rightarrow \R^d$ is a Borel measurable function and $B_t$ is a \mbox{$d$-dimensional} Brownian motion defined on the filtered probability space $(\Omega, \mathcal{F}, \{\mathcal{F}_t\}_{t\in[0,T]}, P)$ where the filtration $\{\mathcal{F}_t\}_{t\in [0,T]}$ is the one generated by $B_t$, $t\in [0,T]$ augmented by the $P$-null sets.

If $b$ is of linear growth and Lipschitz continuous it is well-known that there exists a unique global strong solution to the SDE (\ref{SDE}) which belongs to $\mathbb{D}^{1,2}(\Omega)$. In fact, under more relaxed conditions on $b$ one has the same result, see for instance \cite{Men11}, \cite{Nil13}.

In this section we are concerned with the regularity of the solution in the Malliavin sense in terms of the regularity of $b$. We will assume the following hypotheses for $b$, for every $(t,x) \in [0,T]\times \R^d$
\begin{align}\label{bcondition}
\begin{split}
\sup_{t\in [0,T]}|b(t,x)|\leq C(1+|x|), C>0,\\
D b(t,\cdot), D^2 b(t,\cdot),\dots, D^k b(t,\cdot)\in L^{\infty}(\R^d)
\end{split}\tag{H}
\end{align}
for some $k\geq 1$ where here, the derivatives are understood in the weak sense. In particular, $b$ is $k-1$ times continuously differentiable in virtue of the Sobolev embedding (\ref{embed}) and equation (\ref{SDE}) admits a unique strong solution.

Before we proceed to the main statements of this section we need two preliminary results which are essential for our targets.

\begin{lemm}\label{epsilonbound}
Let $\{b_n\}_{n\geq 0}$ be a sequence of compactly supported smooth functions approximating $b$ a.e. in $t \in [0,T]$ and $x\in \R^d$ such that $\sup_{n\geq 0}|b_n(t,x)| \leq C(1+|x|)$, all $x\in \R^d$ and $t\in [0,T]$. Then for any compact subset $K\subset \R^d$ there exists an $\varepsilon>0$ such that
\begin{align}\label{epsilonboundn}
\sup_{x\in K} \sup_{n\geq 0}E\left[ \mathcal{E}\left( \int_0^T b_n(u,B_u^x)dB_u\right)^{1+\varepsilon} \right] <\infty.
\end{align}
where $B_t^x := x+B_t$ and
$$\mathcal{E}(M_t):= \exp \left\{M_t- M_0 - \frac{1}{2}[M]_t\right\}$$
denotes the Dol\'{e}ans-Dade exponential of a martingale $M$ and $[M]$ the quadratic variation of $M$. In particular we also have
\begin{align}\label{epsilonbound}
\sup_{x\in K} E\left[ \mathcal{E}\left( \int_0^T b(u,B_u^x)dB_u\right)^{1+\varepsilon} \right] <\infty.
\end{align}
\end{lemm}
\begin{proof}
Indeed, write
\begin{align*}
E&\left[\mathcal{E}\left( \int_0^T b_n(u,B_u^x)dB_u \right)^{1+\varepsilon} \right] =\\
&= E\left[\exp \left\{\int_0^T (1+\varepsilon)b_n(u,B_u^x)dB_u - \frac{1}{2}\int_0^T (1+\varepsilon)|b_n(u,B_u^x)|^2du \right\}\right]\\
&= E\Bigg[\exp \Bigg\{\int_0^T (1+\varepsilon)b_n(u,B_u^x)dB_u - \frac{1}{2}\int_0^T (1+\varepsilon)^2|b_n(u,B_u^x)|^2du\\
&+\frac{1}{2}\int_0^T \varepsilon(1+\varepsilon) |b_n(u,B_u^x)|^2du \Bigg\}\Bigg]\\
&= E\left[\exp \left\{\frac{1}{2}\int_0^T  \varepsilon(1+\varepsilon) |b_n(u,X_u^{\varepsilon,x})|^2du \right\}\right]\\
\end{align*}
where the last step follows from Girsanov's theorem and here $X_t^{\varepsilon,x}$ is a solution of the following SDE
$$
\begin{cases}
dX_t^{\varepsilon,x} = (1+\varepsilon)b_n(t,X_t^{\varepsilon,x}) dt + dB_t, \ \ t\in [0,T] \\
X_0^{\varepsilon,x} = x.
\end{cases}
$$
Observe that, since $b$ has at most linear growth, we have
$$
|X_t^{\varepsilon,x}| \leq |x| + C(1+\varepsilon)\int_0^t(1+|X_u^{\varepsilon,x}|)du + |B_t|
$$
for every $t\in [0,T]$. Then Gronwall's inequality gives
\begin{align}\label{gronw}
|X_t^{\varepsilon, x}| \leq \left( |x| + C(1+\varepsilon)T +|B_t|\right)e^{C(1+\varepsilon)T},
\end{align}
and the sublinearity of $b$ and the estimate (\ref{gronw}) give
$$|b_n(u,X_u^{\varepsilon,x})| \leq C_{\varepsilon,T}\left(1+|x| + |B_t|\right)$$
where $C_{\varepsilon,T}$ denotes the collection of all constants depending on $\varepsilon,T$.

As a result,
\begin{align*}
E\Big[ \exp\bigg\{\varepsilon(1+&\varepsilon) \int_0^T |b_n(u,X_u^{\varepsilon,x})|^2 du\bigg\}\Big] \leq E\left[\exp\left\{\tilde{C}_{\varepsilon,T} \int_0^T \left( 1+|x|+|B_u|\right)^2 du\right\}\right]\\
&\leq e^{\tilde{C}_{\varepsilon,T} (1+|x|)^2}E\left[\exp\left\{\tilde{C}_{\varepsilon,T}(1+|x|) \int_0^T (|B_u|+|B_u|^2) du\right\}\right]
\end{align*}
where $\tilde{C}_{\varepsilon,T}>0$ is a constant such that $\lim_{\varepsilon \searrow 0} \tilde{C}_{\varepsilon,T}=0$. Clearly, for every compact set $K\subset \R^d$ we can choose $\varepsilon>0$ small enough such that
$$\sup_{x\in K} \sup_{n\geq 0} E\Big[ \exp\bigg\{\varepsilon(1+\varepsilon) \int_0^T |b_n(u,X_u^{\varepsilon,x})|^2 du\bigg\}\Big] <\infty.$$
\end{proof}

\begin{rem}\label{remdensity}
We point out that the finite dimensional laws of the strong solution of (\ref{SDE}) are absolutely continuous with respect to the Lebesgue measure. To see this, let $A$ denote a set with null Lebesgue measure. Then since $b$ is of, at most, linear growth, by Girsanov's theorem, see e.g. \cite[Proposition 5.3.6]{Kar98}, and Lemma \ref{epsilonbound} one has
\begin{align*}
P(X_t \in A) &\leq E\left[\textbf{1}_{\{B_t^x \in A\}} \mathcal{E} \left( \int_0^T b (u,B_u^x)dB_u\right)\right]\\
&\leq C_{\varepsilon} P(B_t^x \in A)^{\frac{\varepsilon}{1+\varepsilon}}\\
&=0.
\end{align*}
\end{rem}

Next, we give a crucial estimate for the proof of our main results.

\begin{prop} \label{mainEstimate}
Let $B$ be a $d$-dimensional Brownian motion starting from $z_0\in \R^d$ and $b_1, \dots , b_m$ be compactly supported continuously differentiable functions $b_i : [0,T] \times \mathbb{R}^d \rightarrow \mathbb{R}$ for $i=1,2, \dots m$. Let $\alpha_i\in \{0,1\}^d$ be a multiindex such that $| \alpha_i| \leq 1$ for $i = 1,2, \dots, m$. Then there exists a universal constant $C$ (independent of $\{ b_i \}_i$, $m$, and $\{ \alpha_i \}_i$) such that
\begin{equation}
 \label{estimate}
\left| E \left[ \int_{t_0 < t_1 < \dots < t_m < t} \left( \prod_{i=1}^m D^{\alpha_i}b_i(t_i,B_{t_i})  \right) dt_1 \dots dt_m \right] \right| \leq \frac{C^m \prod_{i=1}^m \|b_i \|_{\infty} (t-t_0)^{m/2}}{\Gamma(\frac{m}{2} + 1)}
\end{equation}
for every $t_0,t\in [0,T]$ where $\Gamma$ is the Gamma-function. Here, $\alpha_i=(0,\dots, 1,\dots, 0)$ is a multiindex where $1$ is placed in position $1\leq  j\leq d$ and thus $D^{\alpha_i}$ denotes the partial derivative with respect to the $j'$th space variable.
\end{prop}
\begin{proof}
See \cite[Proposition 3.7]{Men11} for a detailed proof when $|\alpha_i|=1$ for all $i=1,\dots,m$. For the case, $|\alpha_i|\leq 1$ the proof is fairly similar.
\end{proof}

We turn now to one of the main results of this section.

\begin{thm}\label{mainprop}
Let $X_t$, $t\in[0,T]$ denote the solution to equation (\ref{SDE}) with \mbox{$b:[0,T]\times \R^d \rightarrow \R^d$} a function satisfying hypotheses $(H)_1$, i.e. of linear growth with bounded weak derivative, then we have $X_t \in \mathbb{D}^{2,p}(\Omega)$ for all $p\geq 1$. In particular the result holds if $b$ is (globally) Lipschitz.
\end{thm}
\begin{proof}
In order to carry out the proof of Theorem \ref{mainprop}, we use the following result in \mbox{\cite[Proposition 1.5.5.]{Nua10}}.
\begin{prop}
Let $\{X_n\}_{n\geq 0}$ a sequence of random variables such that $X_n \rightarrow X$ in $L^p(\Omega)$, $p\geq 1$ and such that for $k\geq 1$
$$\sup_{n\geq 0}\|X_n\|_{\mathbb{D}^{k,p}(\Omega)}<\infty,$$
then $X\in \mathbb{D}^{k,p}(\Omega)$.
\end{prop}

We start with the proof of Theorem \ref{mainprop} by showing that the solution $X_t$ of (\ref{SDE}) can be approximated by random variables in $L^p(\Omega)$ for every $t\in[0,T]$.

We have that the weak derivative of $b$ lies in $L^{\infty}(\R^d)$ and $b$ has linear growth, i.e. there is $C>0$ such that
$$|b(t,x)|\leq C(1+|x|)$$
for every $t\in [0,T]$ and $x\in \R^d$. Then we can approximate $b$ a.e. in $t\in [0,T]$ and $x\in \R^d$ by a sequence of functions $\{b_n\}_{n\geq 1} \subset \mathcal{C}^{2}(\R^d)$ such that \mbox{$\sup_{n\geq 0} |b_n(t,x)| \leq C(1+|x|)$} and \mbox{$\sup_n \|b_n'\|_{\infty}<\infty$}. For each $t\in[0,T]$, denote by $X_t^n$ the sequence of random variables in $L^p(\Omega)$ solution to equation (\ref{SDE}) with drift coefficient $b_n$. Then

$$X_t^n = x + \int_0^t b_n(u,X_u^n)du + B_t.$$ 

Denote by $p_{X_t}$ the density of $X_t$ for a fixed $t\in [0,T]$ from Remark \ref{remdensity}. Denote by $|\cdot|$ the Euclidean norm in $\R^d$, then
\begin{align*}
E\big[|X_t^n &- X_t|^p\big] = E\left[\bigg| \int_0^t \left( b_n(u,X_u^n) - b(u,X_u)\right) du \bigg|^p\right]\\
&\leq (2t)^{p-1} E\left[\int_0^t |b_n(u,X_u)-b(u,X_u)|^p du\right] + (2t)^{p-1} E\left[\int_0^t |b_n(u,X_u^n)- b_n(u,X_u)|^p du\right]\\
&\leq (2t)^{p-1} E\left[\int_0^t |b_n(u,X_u)-b(u,X_u)|^p du\right] + (2t)^{p-1} \|b_n'\|_{\infty}^p E\left[\int_0^t |X_u^n - X_u|^p du\right].
\end{align*}
Using Gronwall's inequality we obtain
\begin{align*}
E\left[|X_t^n - X_t|^p\right] &\leq (2t)^{p-1} \exp\left\{ (2t)^{p-1}t \sup_{k} \|b_k'\|_{\infty}^p\right\} E\left[\int_0^t |b_n(u,X_u)-b(u,X_u)|^p du\right] \\
&\leq C \int_0^t \int_{\R^d} |b_n(u,x)-b(u,x)|^p p_{X_u}(x) dx du
\end{align*}
for a constant $C>0$ independent of $n$. Then Lebesgue's dominated convergence theorem gives the $L^p(\Omega)$-convergence.

Let us now proceed with the proof that the random variables $X_t^n$ are bounded in $\mathbb{D}^{2,p}(\Omega)$ for every $p\geq 1$: Fix $s_1,t \in [0,T]$, $s_1,\leq t$. Then
\begin{align}\label{malliavin1}
D_{s_1} X_t^n = \mathcal{I}_d + \int_{s_1}^t b_n'(u,X_u^n)D_{s_1}X_u^ndu.
\end{align}
The above equations for $D_{s_1}X_u^n$, $n\geq 1$, are linear equations with matrix-valued unknowns. Since each $b_n$ is smooth we have a unique solution of (\ref{malliavin1}). Again, for notational convenience we denote by \mbox{$\Lambda_m(s,t):= \{(u_1,\dots,u_m)\in [0,T]^m: \ s<u_1<\cdots < u_m <t\}$} the $m$. Then using a Picard iteration argument we may write the solution of (\ref{malliavin1}) as a series expansion as follows
\begin{align}\label{picard}
D_{s_1} X_t^n = \mathcal{I}_d + \sum_{m\geq 1} \int_{\Lambda_m(s_1,t)} b_n'(u_1,X_{u_1}^n)\cdots b_n'(u_m ,X_{u_m}^n) du_1\cdots du_m.
\end{align}
To see that the above expression is indeed the solution of (\ref{malliavin1}) just make the following observation
$$\frac{d}{dt} D_{s_1} X_t^n = b_n'(t,X_t^n)\left(\mathcal{I}_d+ \sum_{m\geq 1} \int_{\Lambda_m(s_1,t)}b_n'(u_1,X_{u_1}^n)\cdots b_n '(u_{m}, X_{u_m}^n)du_1\cdots du_{m}\right).$$

Take now $s_2\in [0,t]$. Then
\begin{align}\label{malliavin21}
D_{s_2} D_{s_1} X_t^n =  \sum_{m\geq 1} \int_{\Lambda_m(s_1 \vee s_2,t)} D_{s_2}\left[b_n'(u_1,X_{u_1}^n)\cdots b_n'(u_m, X_{u_m}^n)\right] du_1\cdots du_m.
\end{align}
We expand the integrand of (\ref{malliavin21}) using Leibniz's rule as follows
$$ D_{s_2}\left[b_n'(u_1,X_{u_1}^n)\cdots b_n'(u_m,X_{u_m}^n)\right] = \sum_{r=1}^m b_n'(u_1,X_{u_1}^n)\cdots b_n''(u_r,X_{u_r}^n)D_{s_2}X_{u_r}^n\cdots b_n'(u_m,X_{u_m}^n).$$

\begin{rem}
We recall here that $b:[0,T]\times \R^d \rightarrow \R^d$ so $Db(t,\cdot): \R^d\rightarrow L(\R^d,\R^d)$ and so $Db(t,x)\in L(\R^d,\R^d)$. The second derivative is then $D^2 b(t,\cdot):\R^d \rightarrow L(\R^d,L(\R^d, \R^d))$ so $D^2b(t,x): L(\R^d, L(\R^d,\R^d)) \cong L^2(\R^d \times \R^d, \R^d)$ denoting by $L^2(\R^d \times \R^d, \R^d )$ the bilinear forms from $\R^d\times \R^d$ into $\R^d$.
\end{rem}

Inserting the representation (\ref{picard}) for $D_{s_2}X_{u_r}^n$ in this case we have that the above quantity can be written as
\begin{align*}
\sum_{r=1}^m & b_n'(u_1,X_{u_1}^n)\cdots b_n''(u_r,X_{u_r}^n)\\
&\times\left(\mathcal{I}_d+ \sum_{m\geq 1} \int_{\Lambda_m(s_2,u_r)} b_n'(v_1, X_{v_1}^n)\cdots b_n'(v_m, X_{v_m}^n) dv_1\cdots dv_m\right)b_n'(u_{r+1},X_{u_{r+1}}^n)\cdots b_n'(u_m, X_{u_m}^n).
\end{align*}
Altogether
\begin{align*}
D_{s_2} D_{s_1} X_t^n &=  \sum_{m_1\geq 1} \int_{\Lambda_{m_1}(s_1\vee s_2,t)} \sum_{r=1}^{m_1} b_n'(u_1, X_{u_1}^n)\cdots b_n''(u_r, X_{u_r}^n)\\
&\times \left(\mathcal{I}_d+ \sum_{m_2\geq 1} \int_{\Lambda_{m_2}(s_2,u_r)} b_n'(v_1, X_{v_1}^n)\cdots b_n'(v_{m_2} ,X_{v_{m_2}}^n) dv_1\cdots dv_{m_2}\right)\\
&\times b_n'(u_{r+1},X_{u_{r+1}}^n) \cdots b_n'(u_{m_1}, X_{u_{m_1}}^n) du_1\cdots du_{m_1}. \\
&= \sum_{m_1\geq 1} \int_{\Lambda_{m_1}(s_1\vee s_2,t)} \sum_{r=1}^{m_1} b_n'(u_1, X_{u_1}^n)\cdots b_n''(u_r, X_{u_r}^n) \cdots b_n'(u_{m_1}, X_{u_{m_1}}^n)du_1\cdots du_{m_1}  \\
&+ \sum_{m_1\geq 1} \int_{\Lambda_{m_1}(s_1\vee s_2,t)} \sum_{r=1}^{m_1} b_n'(u_1, X_{u_1}^n)\cdots  b_n''(u_r, X_{u_r}^n) \\
&\times  \left(\sum_{m_2\geq 1} \int_{\Lambda_{m_2}(s_2,u_r)} b_n'(v_1, X_{v_1}^n)\cdots b_n'(v_{m_2}, X_{v_{m_2}}^n) dv_1\cdots dv_{m_2}\right)\\
&\times b_n'(u_{r+1},X_{u_{r+1}}^n)\cdots b_n'(u_{m_1}, X_{u_{m_1}}^n) du_1\cdots du_{m_1}.
\end{align*}

We reallocate terms by dominated convergence and respecting the order of matrices
\begin{align}\label{malliavin22}
D_{s_2} D_{s_1} & X_t^n = \sum_{m_1\geq 1}\sum_{r=1}^{m_1} \int_{\Lambda_{m_1}(s_1\vee s_2,t)} b_n'(u_1, X_{u_1}^n)\cdots b_n''(u_r, X_{u_r}^n) \cdots b_n'(u_{m_1}, X_{u_{m_1}}^n)du_1\cdots du_{m_1}  \\
&+ \sum_{m_1\geq 1} \sum_{r=1}^{m_1}\sum_{m_2\geq 1} \int_{\Lambda_{m_1}(s_1\vee s_2,t)} \int_{\Lambda_{m_2}(s_2,u_r)} b_n'(u_1, X_{u_1}^n)\cdots  b_n''(u_r, X_{u_r}^n) \notag\\
&\times b_n'(v_1, X_{v_1}^n)\cdots b_n'(v_{m_2}, X_{v_{m_2}}^n) b_n'(u_{r+1},X_{u_{r+1}}^n)\cdots b_n'(u_{m_1}, X_{u_{m_1}}^n) dv_1\cdots dv_{m_2} du_1\cdots du_{m_1}.\notag\\ &=: I_1^n + I_2^n.\notag
\end{align}
Denote by $\|\cdot\|$ the maximum norm on $\R^{d\times d \times d}$. Then Minkowski's inequality gives
\begin{align*}
E\|D_{s_2}D_{s_1}X_t^n\|^p = E\| I_1^n + I_2^n \|^p\leq 2^{p-1}\left(E\|I_1^n\|^p + E\|I_2^n\|^p\right)
\end{align*}
Let $p\geq 1$ and choose $p_1,p_2\in [1,\infty)$ such that $pp_1=2^q$ for some integer $q$ and $\frac{1}{p_1}+\frac{1}{p_2}=1$. We focus now on the term $I_2^n$. Then by Girsanov's theorem we have

\begin{align}
E\|I_2^n\|^p &= E\Bigg[\bigg\| \sum_{m_1\geq 1} \sum_{r=1}^{m_1}\sum_{m_2\geq 1} \int_{\Lambda_{m_1}(s_1\vee s_2,t)} \int_{\Lambda_{m_2}(s_2,u_r)} b_n'(u_1, B_{u_1}^x)\cdots  b_n''(u_r, B_{u_r}^x) \notag\\
&\times b_n'(v_1, B_{v_1}^x)\cdots b_n'(v_{m_2}, B_{v_{m_2}}^x) b_n(u_{r+1}, B_{u_{r+1}}^x)\cdots b_n'(u_{m_1}, B_{u_{m_1}}^x) dv_1\cdots dv_{m_2} du_1\cdots du_{m_1}\bigg\|^p \notag \\
&\times \mathcal{E}\left( \sum_{i=1}^d \int_0^T b_n^{(i)}(u,B_u^x)dB_u^{(i)}\right)\Bigg] \notag.
\end{align}

Then choose $p_2=1+\varepsilon$ and $p_1=\frac{1+\varepsilon}{\varepsilon}$ with $\varepsilon>0$ sufficiently small and apply Lemma \ref{epsilonbound} to obtain 

\begin{align}\label{I2}
\begin{split}
E&\|I_2^n\|^p \leq C_{\varepsilon} \bigg\| \sum_{m_1\geq 1} \sum_{r=1}^{m_1}\sum_{m_2\geq 1} \int_{\Lambda_{m_1}(s_1\vee s_2,t)} \int_{\Lambda_{m_2}(s_2,u_r)} b_n'(u_1,B_{u_1}^x)\cdots  b_n''(u_r,B_{u_r}^x)b_n'(v_1, B_{v_1}^x) \\
&\times \cdots b_n'(v_{m_2}, B_{v_{m_2}}^x) b_n(u_{r+1},B_{u_{r+1}}^x)\cdots b_n'(u_{m_1}, B_{u_{m_1}}^x) dv_1\cdots dv_{m_2} du_1\cdots du_{m_1} \bigg\|_{L^{2^q}(\Omega; \R^{d\times d\times d})}^p 
\end{split}
\end{align}

Now we carry out the product of linear and bilinear forms in the integrand of (\ref{I2}). Recall that $b''(u,B_u^x)=\left(\frac{\partial^2}{\partial x_j \partial x_k} b^{(i)}(u,B_u^x)\right)_{i,j,k=1,\dots,d}$ and $b'(u,B_u^x)=\left(\frac{\partial}{\partial x_j} b^{(i)}(u,B_u^x) \right)_{i,j=1,\dots,d}$ where the superscript $b^{(i)}(u,B_u^x)$ here denotes the $i$-th component of the vector $b(u,B_u^x)$ and $\frac{\partial}{\partial x_j}$, resp. $\frac{\partial^2}{\partial x_j \partial x_k}$, denote the weak derivative of $b^{(i)}(u,B_u^x)$ with respect to the $j$-th space component, resp. with respect to the $j$-th and $k$-th space components. So we represent the second order derivatives as a matrix of matrices in this case, i.e. \mbox{$b''(t,x) = \nabla\otimes \nabla b(t,x)$} where $\otimes$ denotes the Kronecker tensor product.

Hence we can represent the second order derivatives in the integrand in (\ref{I2}) in this manner
\begin{align}\label{D2b}
b''(u,B_u^x)=\begin{pmatrix}
\frac{\partial}{\partial x_1}\begin{pmatrix} \frac{\partial}{\partial x_1} b^{(1)}(u,B_u^x) & \cdots & \frac{\partial}{\partial x_d} b^{(1)}(u,B_u^x) \\ \vdots & \ddots & \vdots \\ \frac{\partial}{\partial x_1} b^{(d)}(u,B_u^x) & \cdots & \frac{\partial}{\partial x_1} b^{(d)}(u,B_u^x)\end{pmatrix}\\
~\\
 \vdots\\
~\\
\frac{\partial}{\partial x_d}\begin{pmatrix} \frac{\partial}{\partial x_1} b^{(1)}(u,B_u^x) & \cdots & \frac{\partial}{\partial x_d} b^{(1)}(u,B_u^x) \\ \vdots & \ddots & \vdots \\ \frac{\partial}{\partial x_1} b^{(d)}(u,B_u^x) & \cdots & \frac{\partial}{\partial x_1} b^{(d)}(u,B_u^x)\end{pmatrix}\\
\end{pmatrix}
\end{align}

The product of $b''(B_{u})$ with $b'(B_v)$ is then
\begin{align}\label{D2bDb}
b''(u,B_u^x)b'(v,B_v^x)=\begin{pmatrix}
\frac{\partial}{\partial x_1}\begin{pmatrix} \frac{\partial}{\partial x_1} b^{(1)}(u,B_u^x) & \cdots & \frac{\partial}{\partial x_d} b^{(1)}(u,B_u^x) \\ \vdots & \ddots & \vdots \\ \frac{\partial}{\partial x_1} b^{(d)}(u,B_u^x) & \cdots & \frac{\partial}{\partial x_1} b^{(d)}(u,B_u^x)\end{pmatrix} b'(v,B_v^x)\\
~\\
 \vdots\\
~\\
\frac{\partial}{\partial x_d}\begin{pmatrix} \frac{\partial}{\partial x_1} b^{(1)}(u,B_u^x) & \cdots & \frac{\partial}{\partial x_d} b^{(1)}(u,B_u^x) \\ \vdots & \ddots & \vdots \\ \frac{\partial}{\partial x_1} b^{(d)}(u,B_u^x) & \cdots & \frac{\partial}{\partial x_1} b^{(d)}(u,B_u^x)\end{pmatrix} b'(v,B_v^x)\\
\end{pmatrix}
\end{align}
As a result
\begin{align*}
& &b''(u,B_u^x)b'(v,B_v^x)= \left( \sum_{l=1}^d \frac{\partial^2}{\partial x_k \partial x_l}b^{(i)}(u,B_u^x)\frac{\partial}{\partial x_j} b^{(l)}(v,B_v^x) \right)_{i,j,k=1}^d
\end{align*}
Hence, taking maximum norm over all products
\begin{align}\label{I2eq2}
\begin{split}
E\|I_2^n\|^p &\leq C_p \ \Bigg( \sum_{m_1\geq 1} \sum_{r=1}^{m_1}\sum_{m_2\geq 1}\sum_{i,j,k=1}^d \sum_{l_1,\dots,l_{m_1+m_2-1}=1}^d \bigg\| \int_{\Lambda_{m_1}(s_1\vee s_2,t)}\int_{\Lambda_{m_2}(s_2,u_r)} \frac{\partial}{\partial x_{l_1}}b_n^{(i)}(u_1,B_{u_1}^x) \\
&\times \frac{\partial}{\partial x_{l_2}}b_n^{(l_1)}(u_2,B_{u_2}^x)\cdots  \frac{\partial}{\partial x_{l_{r-1}}}b_n^{(l_{r-2})}(u_{r-1},B_{u_{r-1}}^x)\frac{\partial}{\partial x_k}\frac{\partial}{\partial x_{l_r}}b_n^{(l_{r-1})}(u_r,B_{u_r}^x) \\
&\times \frac{\partial}{\partial x_{l_{r+1}}}b_n^{(l_r)}(v_1,B_{v_1}^x)\cdots  \frac{\partial}{\partial x_{ l_{r+m_2} } }b_n^{ (l_{r+m_2-1} ) }(v_{m_2},B_{v_{m_2}}^x)\frac{\partial}{\partial x_{l_{r+m_2+1}}}b_n^{(l_{r+m_2})}(u_{r+1},B_{u_{r+1}}^x)\\
&\times  \cdots  \frac{\partial}{\partial x_j}b_n^{(l_{m_1+m_2-1})}(u_{m_1},B_{u_{m_1}}^x)  dv_1\cdots dv_{m_2} du_1\cdots du_{m_1}  \bigg\|_{L^{2^q}(\Omega; \R)}\Bigg)^p. 
\end{split}
\end{align}
Observe the second order partial derivatives in the integrand.

The following step is to apply expectation and get rid of the second order derivatives. To do so, we will use the estimate from Proposition \ref{mainEstimate}.

Before applying Proposition \ref{mainEstimate} we need to make the following observation on the integrating regions in connection to (\ref{I2eq2}): the iterated integrals of (\ref{I2eq2}) can be split up as a sum of integrals where the regions which we integrate over are ordered. Indeed, using formula \eqref{shuffle2} we express the term in (\ref{I2eq2}) as follows
\begin{align}\label{I2eq3}
\begin{split}
E&\|I_2^n\|^p\\
\leq&  C_p \ \Bigg( \sum_{m_1\geq 1} \sum_{r=1}^{m_1}\sum_{m_2\geq 1}\sum_{i,j,k=1}^d \sum_{l_1,\dots,l_{m_1+m_2-1}=1}^d \sum_{\sigma \in S_r(m,n)}\bigg\| \int_{\Lambda_{m_1+m_2}(s_1\vee s_2,t)} \frac{\partial}{\partial x_{l_1}}b_n^{(i)}(w_{\sigma(1)},B_{w_{\sigma(1)}}^x) \\
&\times \frac{\partial}{\partial x_{l_2}}b_n^{(l_1)}(w_{\sigma(2)},B_{w_{\sigma(2)}}^x)\cdots  \frac{\partial}{\partial x_{l_{r-1}}}b_n^{(l_{r-2})}(w_{\sigma(r-1)},B_{w_{\sigma(r-1)}}^x)\frac{\partial}{\partial x_k}\frac{\partial}{\partial x_{l_r}}b_n^{(l_{r-1})}(w_{\sigma(r)},B_{w_{\sigma(r)}}^x) \\
&\times \frac{\partial}{\partial x_{l_{r+1}}}b_n^{(l_r)}(w_{\sigma(r+1)},B_{w_{\sigma(r+1)}}^x)\cdots  \frac{\partial}{\partial x_{ l_{r+m_2} } }b_n^{ (l_{r+m_2-1} ) }(w_{\sigma(r+m_2)},B_{w_{\sigma(r+m_2)}}^x)\\
&\times\frac{\partial}{\partial x_{l_{r+m_2+1}}}b_n^{(l_{r+m_2})}(w_{\sigma(r+m_2+1)},B_{w_{\sigma(r+m_2+1)}}^x) \cdots \\
&  \cdots  \frac{\partial}{\partial x_j}b_n^{(l_{m_1+m_2-1})}(w_{\sigma(m_1+m_2)},B_{w_{\sigma(m_1+m_2)}}^x)  dw_1\cdots dw_{m_1+m_2} \bigg\|_{L^{2^q}(\Omega; \R)}\Bigg)^p. 
\end{split}
\end{align}

Now that the sets over which we integrate are symmetric we can use deterministic integration by parts to write the integrals in (\ref{I2eq3}) to the power two as a sum of at most $2^{2m}$ summands of the form
$$\int_{\Lambda_{2m}^s (s_1\vee s_2,t)} g_1(w_1)\cdots g_{2m}(w_{2m})dw_1\cdots dw_{2m}$$
where $m:=m_1+m_2$ and $g_l \in \left\{ \frac{\partial}{\partial x_j} b^{(i)}(\cdot, B_{\cdot}^x), \frac{\partial^2}{\partial x_{l_1} \partial x_{l_2}} b^{(k)}(\cdot,B_{\cdot}^x), i,j,k,l_1,l_2=1,\dots d\right\}$ and \mbox{$l=1,\dots,2m$}. Once more, we can write the integrals to the power four as a sum of at most $2^{8m}$ summands of the form
$$\int_{\Lambda_{4m}^s (s_1\vee s_2,t)} g_1(w_1)\cdots g_{4m}(w_{4m})dw_1\cdots dw_{4m}.$$
Repeating this principle, one can write the integrals to the power $2^q$ as a sum of at most $2^{q2^q m}$ summands of the form
$$ \int_{\Lambda_{2^q m}^s (s_1\vee s_2,t)} g_1(w_1)\cdots g_{2^q m}(w_{2^q m})dw_1\cdots dw_{2^q m}.$$

Combining this with Proposition \ref{mainEstimate} we obtain
\begin{align}\label{I2eq4}
\begin{split}
&\bigg\| \int_{\Lambda_{m_1+m_2}(s_1\vee s_2,t)} \frac{\partial}{\partial x_{l_1}}b_n^{(i)}(w_{\sigma(1)},B_{w_{\sigma(1)}}^x) \\
&\times \frac{\partial}{\partial x_{l_2}}b_n^{(l_1)}(w_{\sigma(2)},B_{w_{\sigma(2)}}^x)\cdots  \frac{\partial}{\partial x_{l_{r-1}}}b_n^{(l_{r-2})}(w_{\sigma(r-1)},B_{w_{\sigma(r-1)}}^x)\frac{\partial}{\partial x_k}\frac{\partial}{\partial x_{l_r}}b_n^{(l_{r-1})}(w_{\sigma(r)},B_{w_{\sigma(r)}}^x) \\
&\times \frac{\partial}{\partial x_{l_{r+1}}}b_n^{(l_r)}(w_{\sigma(r+1)},B_{w_{\sigma(r+1)}}^x)\cdots  \frac{\partial}{\partial x_{ l_{r+m_2} } }b_n^{ (l_{r+m_2-1} ) }(w_{\sigma(r+m_2)},B_{w_{\sigma(r+m_2)}}^x)\\
&\times\frac{\partial}{\partial x_{l_{r+m_2+1}}}b_n^{(l_{r+m_2})}(w_{\sigma(r+m_2+1)},B_{w_{\sigma(r+m_2+1)}}^x) \cdots \\
&  \cdots  \frac{\partial}{\partial x_j}b_n^{(l_{m_1+m_2-1})}(w_{\sigma(m_1+m_2)},B_{w_{\sigma(m_1+m_2)}}^x)  dw_1\cdots dw_{m_1+m_2} \bigg\|_{L^{2^q}(\Omega; \R)}  \\
&\leq \left(\frac{2^{q2^q (m_1+m_2)}C_{d,p,T}^{2^q (m_1+m_2)} \|b_n'\|_{\infty}^{2^q (m_1+m_2)}|t-(s_1 \vee s_2)|^{2^{q-1}(m_1+m_2)}}{\Gamma\left( 2^{q-1}(m_1+m_2)+1\right)}\right)^{2^{-q}} \\
&=\frac{2^{q(m_1+m_2)}C_{d,p,T}^{m_1+m_2}\|b_n'\|_{\infty}^{m_1+m_2}|t-(s_1 \vee s_2)|^{(m_1+m_2)/2}}{\left[\left(2^{q-1}(m_1+m_2)\right)!\right]^{2^{-q}}}
\end{split}
\end{align}

Using the bound in (\ref{I2eq4}) we get
\begin{align*}
E&\|I_2^n\|^p \leq\\
&\leq \Bigg( \sum_{m_1\geq 1} \sum_{r=1}^{m_1}\sum_{m_2\geq 1} C^{m_1+m_2}\frac{d^{m_1+m_2+2}2^{q(m_1+m_2)}C_{d,p,T}^{m_1+m_2} \|b_n'\|^{m_1+m_2}|t-(s_1 \vee s_2)|^{(m_1+m_2)/2}}{\left[\left(2^{q-1}(m_1+m_2)\right)!\right]^{2^{-q}}} \Bigg)^p\\
&\leq \Bigg( \sum_{m_1\geq 1}\sum_{m_2\geq 1} m_1C^{m_1+m_2}\frac{d^{m_1+m_2+2}2^{q(m_1+m_2)}C_{d,p,T}^{m_1+m_2} \|b_n'\|^{m_1+m_2}|t-(s_1 \vee s_2)|^{(m_1+m_2)/2}}{\left[\left(2^{q-1}(m_1+m_2)\right)!\right]^{2^{-q}}} \Bigg)^p\\
&\leq \Bigg( \sum_{m\geq 1} mC^m\frac{d^{m+2}2^{qm}C_{d,p,T}^{m} \|b_n'\|_{\infty}^{m}|t-(s_1 \vee s_2)|^{m/2}}{\left[\left(2^{q-1}m\right)!\right]^{2^{-q}}} \Bigg)^p\\
&\leq C_{d,p,T} f(\|b_n'\|_{\infty})
\end{align*}
for some continuous function $f$ only depending on $d$, $p$ and $T$. As a result,
$$\sup_{n\geq 0}\sup_{s_1,s_2\in [0,T]} E\|I_2^n\|^p \leq C_{d,p,T} \sup_{n\geq 0}\sup_{s_1,s_2\in [0,T]} f(\|b_n'\|_{\infty}) < \infty.$$

Finally, one can bound $E\|I_1^n\|^p$ using exactly the same steps as for $I_2^n$.
\end{proof}

We are now in a position to state one of the main results of this section on the Malliavin regularity of the solution to SDE (\ref{SDE}).

\begin{thm}\label{mainthm}
Assume that $b$ satisfies condition (\ref{bcondition}) for some $k\geq 1$. Let $X_t$, $t\in[0,T]$ denote the solution to equation (\ref{SDE}). Then
$$X_t \in \bigcap_{p\geq 1} \mathbb{D}^{k+1,p}(\Omega).$$
\end{thm}
\begin{proof}
The proof of this more general result relies on Theorem \ref{mainprop} by iterating all arguments up to $k+1$. Similarly as before, let $\{b_n\}_{n\geq 1} \subset \mathcal{C}^{k+1}(\R^d)$ be an approximating sequence of functions such that $b_n \to b$ a.e. in $t\in [0,T]$ and $x\in \R^d$ as $n\to \infty$ and $\sup_{n\geq 0} |b_n(t,x)|\leq C(1+|x|)$ and $\sup_n \|b_n^{(j)}\|<\infty$, $j\leq k$ where $\|\cdot\|$ denotes any norm in $\R^d \times \overset{j)}{\dots} \times \R^d$, $j\leq k$. For each $t\in[0,T]$, denote by $X_t^n$ the sequence of random variables in $L^p(\Omega)$ solution to equation (\ref{SDE}) with drift coefficient $b_n$. Then we wish to compute the Malliavin derivative of $X_t^n$ up to order $k+1$. This becomes a large expression where the terms increase at a binomial speed. We saw in the proof of Proposition (\ref{mainprop}) that the second order Malliavin derivative of $X_t^n$ can be written as $D_{s_2}D_{s_1}X_t^n = I_1^n+I_2^n$ where the integrals in $I_2^n$ are doubled. If we fix $s_3\in[0,t]$ then $D_{s_3}D_{s_2}D_{s_1} X_t^n = D_{s_3} I_1^n + D_{s_3} I_2^n = I_1^n + I_2^n + I_3^n + I_4^n$ and so on. Each term $I_i^n$, $i=1,2,3,4$ is a sum of integrals of the form (\ref{malliavin22}) with at most one factor $b_n^{(3)}$. Iterating this argument, we have that for fixed $s_1,\dots,s_{k+1}\in [0,t]$
$$D_{s_{k+1}}\cdots D_{s_1}X_t^n = I_1^n + \dots + I_{2^{k}}^n$$
where each $I_i^n$, $i=1,2,\dots, 2^k$ is an integral over at most $\Lambda_{m_1+\dots+m_{k+1}}$ with at most one factor $b_n^{(k+1)}$ and the rest $b_n^{(j)}$, $j \leq k$. This can be readily checked by looking at expression (\ref{malliavin22}). Then, estimating $I_{2^{k}}^n$ implies that all former terms are also bounded.

To illustrate $I_{2^{k}}^n$ we use expression (\ref{malliavin22}) and apply $D_{s_3}\cdots D_{s_{k+1}}$ and focus on the last term. In order to simplify notation and make the reading clearer we consider indices $m_1,\dots,m_{k+1},r_1,\dots, r_{k}\in \mathbb{N}\setminus \{0\}$ and denote
$$\sum_{\substack{m_1,\dots,m_{k+1}\\r_1,\dots,r_{k}}} :=\sum_{m_1\geq 1}\sum_{r_1=1}^{m_1} \sum_{m_2\geq 1}\sum_{r_2=1}^{m_1+m_2}\cdots \sum_{r_{k}=1}^{m_1+\cdots+m_{k}}\sum_{m_{k+1}\geq 0},$$
as well as,
$$\int_{\Delta} := \int_{\Lambda_{m_1}(s_1\vee\dots\vee s_{k+1},t)}\int_{\Lambda_{m_2}(s_2\vee\dots\vee s_{k+1},t)}\cdots \int_{\Lambda_{m_{k+1}}(s_{k+1},t)}.$$
Then $I_{2^{k}}^n$ will take the following form
\begin{align*}
I_{2^{k}}^n &= \sum_{\substack{m_1,\dots,m_{k+1}\\r_1,\dots,r_{k}}} \int_{\Delta} \mathcal{A}(u_1^1,\dots,u_{m_1}^1,\dots, u_{1}^{k+1},\dots, u_{m_{k+1}}^{k+1}) du_1^{k+1}\cdots du_{m_{k+1}}^{k+1} \cdots du_1^1 \cdots du_{m_1}^1
\end{align*}
with integrand
\begin{align*}
\mathcal{A}:= g_n(u_1^1)\cdots g_n(u_{r_1}^1)\bigg[g_n(u_1^2)\cdot g_n(u_{r_2}^2)&\bigg[\cdots g_n(u_1^{k+1})\cdots\\
&\cdots g_n(u_{m_{k+1}}^{k+1})\bigg]g_n(u_{r_{k}+1}^{k})\cdots g_n(u_{m_2}^2)\bigg]g_n(u_{r_1+1}^1)\cdots g_n(u_{m_1}^1)
\end{align*}
where the functions $g_n$ denote an element in the set
$$g_n\in \{D b_n, D^2 b_n, \cdots, D^{k+1} b_n\}.$$
Then, using exactly the same procedure as for $I_1^n$ and $I_2^n$, mutatis mutandis, we obtain an integral of products of partial derivatives of at most order $k+1$, this together with Proposition \ref{mainEstimate} one is able to get rid of the $k+1$-th derivative as we did for $I_2^n$ in Theorem \ref{mainthm}.
\end{proof}

To emphasize that the solution depends on the initial point $x$ we write $X_t^x$. Next result gives a condition for the regularity of $x\mapsto X_t^x$ in the space variable in the Sobolev sense.

\begin{thm}\label{sobreg}
Assume that $b$ satisfies condition (\ref{bcondition}) for some $k\geq 1$. Let $U\subset \R^d$ be an open bounded set and $X_t$, $t\in[0,T]$ denote the solution to equation (\ref{SDE}). Then
$$X_t^{\cdot} \in \bigcap_{p\geq 1} L^2\big(\Omega, W^{k+1,p}(U)\big).$$
\end{thm}
\begin{proof}
This result actually follows by observing that the process $\frac{\partial}{\partial x} X_t^x$ satisfies the following linear ODE
$$\frac{\partial}{\partial x} X_t^x = \mathcal{I}_d + \int_0^t b'(u, X_u^x)\frac{\partial}{\partial x}X_u^x du.$$

This equation is the same as (\ref{malliavin1}) when $s=0$. Using this observation, in connection with the same method employed in the proof of Theorem \ref{mainthm} by replacing the Malliavin derivative of $X_t$ with $\frac{\partial}{\partial x}X_t^x$ we get that for the approximating sequence of solutions $X_t^{n,x}$, $n\geq 0$ described in Theorem \ref{mainthm} we have
$$\sup_{n\geq 0}\sup_{x\in \overline{U}} E\left[\|\frac{\partial^{j}}{\partial x^{j}} X_t^{n,x}\|^p \right] <\infty$$
for all $j=0,\dots,k+1$ and any $p\geq 1$ so $X_t^{n,\cdot}$ is bounded in the Sobolev norm $ L^2(\Omega, W^{k+1,p}(U))$ for each $n\geq 0$. Indeed
$$\sup_{n\geq 0} \|X_t^{n,\cdot}\|_{L^2(\Omega, W^{k+1,p}(U))}^2 = \sup_{n\geq 0}\sum_{i=0}^{k+1} E\left[\|\frac{\partial^i}{\partial x^i}X_t^{n,\cdot}\|_{L^p(U)}^2\right]\leq \sum_{i=0}^{k+1} \int_{U} \sup_{n\geq 0}E\left[ \|\frac{\partial^i}{\partial x^i}X_t^{n,x}\|^p \right]dx <\infty.$$
Since $L^2(\Omega, W^{k+1,p}(U))$ is reflexive, by Banach-Alaoglu's theorem we get that the set $\{X_t^{n,x}\}_{n\geq 0}$ is weakly compact in the $L^2(\Omega, W^{k+1,p}(U))$ topology. Thus, there exists a subsequence $n(j)$, $j\geq 0$ such that
$$X_t^{n(j),\cdot}  \xrightarrow[j\to \infty]{w} Y \in L^2(\Omega, W^{k+1,p}(U)).$$

On the other hand, we have that $X_t^{n,x} \to X_t^x$ strongly in $L^p(\Omega)$, so by uniqueness of the limit we can conclude that
$$X_t^{\cdot} =Y, \ \ P-a.s.$$
\end{proof}

\begin{rem}
The previous result actually gives classical derivatives of the solution up to order $k+\alpha$ with $\alpha\in (0,1)$ as a consequence of the Sobolev embedding (\ref{embed}).
\end{rem}

\section{Application to the regularity of densities}

As mentioned in the introduction one implication of improving the Malliavin regularity of SDEs with drift coefficient satisfying hypotheses \textbf{(H)} is that the finite dimensional laws have $k$-times differentiable densities due to a result by V.Bally and L.Caramellino, see \cite{Bally2011}. We see this as an improvement of the regularity condition given by \cite{Kus82} for the additive noise case. In addition, we see that the boundedness of $b$ is not needed.

The following is a consequence of Theorem \ref{mainthm} for the special case $d=1$ and illustrates how we may gain regularity of the densities of solutions to (\ref{SDE}) and provide with an explicit expression for the density and its derivatives. Later on, we will show it for higher dimensions.

\begin{cor}
For $d=1$, let $p_{X_t}$ denote the density of the solution $X_t$ to equation (\ref{SDE}) for a given \mbox{$t\in [0,T]$}. If $b$ satisfies (\ref{bcondition}) for some $k\geq 1$ then $p_{X_t} \in C^{k-1}(\R)$ for every $t\in [0,T]$.
\end{cor}
\begin{proof}
Let $G_0, G_1,G_2,\dots,G_k$ be the random variables defined as $G_0=1$ and for each $i=1,\dots,k$
$$G_i=\delta \left(G_{i-1}\cdot\left(\int_0^T D_s X_t ds\right)^{-1}\right),$$
where $\delta$ denotes the Skorokhod integral as introduced in \eqref{skorokhod}.

It is known that if $X_t\in \mathbb{D}^{1,2}(\Omega)$, $\int_0^T D_s X_t ds \neq 0$, $P-a.s.$ and \mbox{$G_i\left( \int_0^T D_s X_t ds\right)^{-1}\in \mbox{Dom}(\delta)$} for each $i=0,\dots,k$ then $X_t$ has a density of class $C^k(\R)$ and
\begin{align}\label{density}
\frac{d^i}{dy^i}p_{X_t}(y) = (-1)^i E\left[ \textbf{1}_{\{X_t>y\}} G_{i+1}\right]
\end{align}
for each $i=0,\dots,k$. See \cite[p115]{Nua10}.

We will then prove that $G_i \left(\int_0^T D_s X_t ds\right)^{-1}\in \mbox{Dom}(\delta)$ for $i=0,\dots,k-1$. First, observe that for dimension $d=1$ we can easily solve the linear SDE for $D_s X_t$ and write
\begin{align}\label{mallexp}
D_s X_t = \exp\left\{\int_s^t b'(u,X_u)du\right\}
\end{align}
where $b'$ denotes the weak derivative of $b$ (one may also use local time to express (\ref{mallexp}) independently of $b'$ if $b$ is non-regular, see \cite{Ein2000}). Hence, for any $t\in [0,T]$ there is an $\varepsilon >0$ such that \mbox{$\int_0^T D_s X_tds \geq \varepsilon >0$}. Since $x\mapsto \frac{1}{x}$ is smooth on the domain $(\varepsilon, \infty)$ we see that \mbox{$\left(\int_0^T D_s X_t ds\right)^{-1} \in \mathbb{D}^{k,2}(\Omega)$} since $X_t \in \mathbb{D}^{k+1,2}(\Omega)$ by Theorem \ref{mainthm}. Denote $F:=\left(\int_0^T D_s X_t ds\right)^{-1}$.

Now, since $F\in \mathbb{D}^{1,2}(\Omega)$ we have $F\in \mbox{Dom}(\delta)$ and $G_1=\delta(F) = FW(T) + \int_0^T D_s F ds$. Then we see that $G_1\in \mathbb{D}^{1,2}(\Omega)$ and hence $G_1F\in \mathbb{D}^{1,2}(\Omega)$ therefore $G_1 F\in \mbox{Dom}(\delta)$ with \mbox{$G_2= \delta(G_1F)= G_1FW(T) - \int_0^T [D_s G_1 F + G_1 D_sF]ds$}. Again, it is readily checked that \mbox{$G_2\in \mathbb{D}^{1,2}(\Omega)$} since \mbox{$G_1,F\in \mathbb{D}^{1,2}(\Omega)$} so $G_2F\in \mbox{Dom}(\delta)$. For a fixed $i=0,\dots,k-1$ we have $G_i,F\in \mathbb{D}^{1,2}(\Omega)$ therefore $G_iF\in \mbox{Dom}(\delta)$ with $G_{i+1}=\delta(G_iF)= G_iFW(T) - \int_0^T [D_sG_i F + G_i D_sF]ds$. So $G_i$ is well-defined for $i=0,\dots,k$ but we can not say anything about $G_{k+1}$ so $p_{X_t}$ is at least $k-1$-times differentiable with derivatives given by (\ref{density}).
\end{proof}

As a consequence of the Malliavin regularity we have shown for SDEs of the form (\ref{SDE}) we may apply the results by V.Bally and L.Caramellino, see \cite{Bally2011}, to be able to obtain regularity of the densities, also in higher dimension. In order to do so, we need to study integrability properties of the Malliavin covariance matrix. Let us denote
$$\gamma_{X_t}^{ij} := \langle D_{\cdot} X_t^{(i)}, D_{\cdot} X_t^{(j)}\rangle_{H}, \ \ i,j=1,\dots,d,$$
the Malliavin covariance matrix of the process $X_t$, given $t\in [0,T]$. We will say that \mbox{$\gamma_{X_t} = (\gamma_{X_t}^{ij})_{i,j=1,\dots,d}$} satisfies the \emph{non-degeneracy condition} whenever
\begin{align}\label{nondeg}
(\det \gamma_{X_t})^{-1} \in \bigcap_{p\geq 1} L^p(\Omega).
\end{align}

Next, we invoke a result by \cite[Proposition 23]{Bally2011} which gives us the desired properties on the density of $X_t$, $t\in [0,T]$.

\begin{prop}
Let $F=(F^1,\dots,F^d)$ with $F^1,\dots,F^d \in \bigcap_{p\geq 1} \mathbb{D}^{k+1,p}(\Omega)$. Assume that condition (\ref{nondeg}) holds for $\gamma_F$. Denote by $p_F$ the density of $F$. Then $p_F \in C^{k-1,\alpha}(\R^d)$ with $\alpha<1$, i.e. $p_F$ is $k-1$-times differentiable with H\"{o}lder continuous derivatives of exponent $\alpha<1$.
\end{prop}

In view of the above result we only need to check that the non-degeneracy condition (\ref{nondeg}) is fulfilled. To do so, we use the following intermediate result.

\begin{lemm}\label{lemmaepsilon}
Let $Z:\Omega \rightarrow E$ be a random variable taking values on a separable Banach space with norm $\|\cdot\|_E$. Fix $p>0$. Then the following are equivalent
\begin{itemize}
\item[(i)]
\begin{align}
E\left[\|Z\|_E^{-p}\right] <\infty.
\end{align}
\item[(ii)] There exists $\varepsilon_0 >0$, depending on $p$, such that
$$\int_0^{\varepsilon_0} \varepsilon^{-(p+1)} P(\|Z\|_E^2< \varepsilon) d\varepsilon <\infty.$$
\end{itemize}
\end{lemm}
\begin{proof}
We have that, for any positive integrable random variable $Y$,
$$E[Y] = \int_0^\infty P(Y>\eta) d\eta.$$

Condition $(i)$ implies that $\|Z\|_E >0$ $P$-a.s. so
\begin{align*}
E\left[\|Z\|_E^{-2p}\right] &= \int_0^{\eta_0} P(\|Z\|_E^{-2p}>\eta) d\eta  + \int_{\eta_0}^\infty P(\|Z\|_E^{-2p}>\eta) d\eta\\
&\leq \eta_0 +  \int_{\eta_0}^\infty P(\|Z\|_E^{-2p}>\eta) d\eta\\
&= \eta_0 + p\int_0^{\eta_0^{-1/p}} \varepsilon^{-(p+1)} P(\|Z\|_E^2 <\varepsilon)d\varepsilon
\end{align*}
where in the last step we have used the change of variables $\eta = \varepsilon^{-p}$.
\end{proof}

Now we are in a position to prove the non-degeneracy condition for the Malliavin matrix associated to the solution of the SDE (\ref{SDE}). The proof of this result is much inspired in Proposition 8.1 from \cite{MSS05}.
\begin{prop}
Let $X_t$, $t\in [0,T]$ be the solution to SDE (\ref{SDE}) with drift coefficient $b$ satisfying condition (\ref{bcondition}) for $k=1$. Then the Malliavin covariance matrix $\gamma_{X_t}$ satisfies
$$(\det \gamma_{X_t})^{-1} \in \bigcap_{p\geq 1} L^p(\Omega)$$
\end{prop}
\begin{proof}
Consider $X_t^n$ with drift coefficient $b_n$ approximating $b$ a.e. such that $\sup_{n\geq 0} \|b_n'\|_{\infty}<\infty$.

It suffices to show that
$$\sup_{n\geq 0} E \left[\bigg| \int_0^T \|D_s X_t^n\|_{\infty}^2 ds \bigg|^{-p}\right]<\infty$$
for any $p\geq 1$.

Recall that for $0\leq s\leq t$, $t\in [0,T]$ we have
$$D_s X_t^n = \mathcal{I}_d + \sum_{m\geq 1} \int_{s<u_1<\cdots <u_m<t} b_n'(u_1,X_{u_1}^n) \cdots b_n'(u_m, X_{u_m}^n)du_1 \cdots du_m.$$

Then for any $\delta >0$, $t-\delta >0$ one has
\begin{align*}
\int_0^T \|D_s X_t^n \|_\infty^2 ds \geq \int_{t-\delta}^t \|D_s X_t^n\|_\infty^2 ds \geq \frac{\delta}{2} - I_n(t,\delta)
\end{align*}
where
$$I_n(t,\delta):= \int_{t-\delta}^t \bigg\|\sum_{m\geq 1} \int_{s<u_1<\cdots <u_m<t} b_n'(u,X_{u_1}^n)\cdots b_n'(u,X_{u_m}^n)du_1\cdots du_m\bigg\|_{\infty}^2ds.$$

Clearly, we have
\begin{align}\label{deltap}
\sup_{n\geq 0} E\left[|I_n(t,\delta)|^p\right] \leq C \delta^p
\end{align}
since $b_n'$, $n\geq 0$ are uniformly bounded.

Then by the previous estimates
\begin{align*}
P\left(\|D_{\cdot} X_t^n\|_{L^2(\Omega, \R^{d\times d})}^2 <\varepsilon\right) &\leq P\left( \int_{t-\delta}^t \|D_s X_t^n\|_\infty^2 ds <\varepsilon\right) \\
&\leq P\left( I_n(t,\delta) \geq \frac{\delta}{2} - \varepsilon\right)\\
&\leq \left(\frac{\delta}{2}-\varepsilon\right)^{-p} E[|I_n(t,\delta)|^p]
\end{align*}
for any $p\geq 1$ due to Chebyshev's inequality. Now, by estimate (\ref{deltap}) we obtain that
\begin{align*}
\sup_{n\geq 0} P\left(\|D_{\cdot} X_t^n\|_{L^2(\Omega, \R^{d\times d})}^2 <\varepsilon\right)&\leq C \left(\frac{\delta}{2}-\varepsilon\right)^{-p}\delta^p.
\end{align*}

By virtue of Lemma \ref{lemmaepsilon} we can conclude if we find $\delta:(0,\infty) \rightarrow \R$, $\varepsilon \mapsto \delta(\varepsilon)$ such that $\lim_{\varepsilon \searrow 0} \delta(\varepsilon) = 0$ and
$$\int_0 \varepsilon^{-(p+1)}\left(\frac{\delta (\varepsilon)}{2}-\varepsilon\right)^{-p}\delta (\varepsilon)^p d\varepsilon < \infty$$
for an arbitrary large $p\geq 1$.

We claim that
$$\delta (\varepsilon) := \left|\frac{2\varepsilon^{\frac{1}{2p}+2}}{\varepsilon^{\frac{1}{2p}+1}-2}\right|$$
does the job.
\end{proof}

Finally, we are able to state our conditions to determine the regularity of densities of solutions to SDEs.

\begin{cor}\label{maincor}
Let $X_t^x$, $t\in [0,T]$ be the strong solution to SDE (\ref{SDE}). Assume $b$ satisfies condition (\ref{bcondition}) for some integer $k\geq 1$. Then the density $p_{X_t}$ belongs to $C^{k-1,\alpha}(\R^d)$, $\alpha<1$, i.e. $k-1$-times continuously differentiable with H\"{o}lder continuous derivatives with exponent $\alpha<1$.
\end{cor}

We end this section by giving an example that shows that the Malliavin regularity we obtained in Theorem \ref{mainthm} is optimal when $k=1$, for the general condition we conjecture it is also optimal.

\begin{exam}
\normalfont
In this example we show that Theorem \ref{mainprop} is an optimal result in the sense that, if $b$ is of linear growth and one time weakly differentiable with bounded derivative then $X_t\in \mathbb{D}^{2,p}(\Omega)$ for all $p\geq 1$ and $X_t\notin \mathbb{D}^{3,p}(\Omega)$ for any $p\geq 1$. Just choose $b$, in dimension $d=1$, to be such that
$$b'(x) = \textbf{1}_{(0,\infty)}(x), \ x\in \R.$$
Then fix $t\in [0,T]$ and for $s_1\leq t$
$$D_{s_1} X_t = \exp\left\{\int_{s_1}^t b'(X_u)du\right\}.$$
Denote by $\tilde{b}(x) := b(a) + \int_a^x b(y)dy$, $a\in \R$ a primitive of $b$. It\^{o}'s formula implies
$$D_{s_1} X_t = \exp\left\{2\tilde{b}(X_t) - 2\tilde{b}(X_{s_1})-2\int_{s_1}^t b^2(X_u)du-2\int_{s_1}^t b(X_u)dB_u\right\}.$$
Then by Theorem \ref{mainthm}, $D_{s_1}X_t \in \mathbb{D}^{1,2}(\Omega)$ for all $t\in [0,T]$. So for $s_2\leq t$
\begin{align}\label{twoMall}
\begin{split}
D_{s_2}D_{s_1} X_t &=  D_{s_1}X_t \bigg(2b(X_t)D_{s_2}X_t - 2b(X_{s_1}) D_{s_2}X_{s_1})-4\int_{s_1\vee s_2}^t b(X_u)b'(X_u)D_{s_2}X_u du -2 b(X_{s_2})\\
&-2\int_{s_1\vee s_2}^t b'(X_u)D_{s_2}X_u dB_u ) \bigg).
\end{split}
\end{align}
Now observe that $b(X_t)\in \cap_{p\geq 1} \mathbb{D}^{1,p}(\Omega)$ for all $t\in [0,T]$ so all terms are \mbox{immediately} Malliavin differentiable with all moments except from maybe $\int_{s_1\vee s_2}^t b(X_u)b'(X_u)D_{s_2}X_u du$ and $\int_{s_1\vee s_2}^t b'(X_u)D_{s_2}X_u dB_u$. The stochastic integral is in fact not Malliavin differentiable. Indeed, by \cite[Lemma 1.3.4]{Nua10}
\begin{align}\label{stochint}
\int_{s_1\vee s_2}^t b'(X_u)D_{s_2}X_u dB_u \in \mathbb{D}^{1,2}(\Omega)
\end{align}
if, and only if
$$b'(X_u)D_{s_2}X_u \in \mathbb{D}^{1,2}(\Omega).$$ 
On the other hand we have $b'(X_u)=\textbf{1}_{(0,\infty)}(X_u) \notin \mathbb{D}^{1,2}(\Omega)$ since $0<P(0<X_u<\infty)<1$, see \mbox{\cite[Proposition 1.2.6]{Nua10}}, and so $D_{s_1}X_t\int_{s_1\vee s_2}^t b'(X_u)D_{s_2}X_u dB_u \notin \mathbb{D}^{1,2}(\Omega)$.

Let us finally prove that
$$Y_t := \int_{s_1\vee s_2}^t b(X_u)b'(X_u)D_{s_2}X_u du \in \mathbb{D}^{1,2}(\Omega).$$
Let $\{b_n\}_{\{n\geq 0\}}$ be a sequence of smooth functions such that $b_n(x) \to b(x)$ a.e. in $x\in \R$ as $n\to \infty$ and $\sup_{n\geq 0}\|b_n'\|_{\infty}<\infty$ and $b_n(x),b_n'(x),b_n''(x)\geq 0$ for all $x\in \R$, we claim that this is trivially possible by the very concrete shape of the function $b$ in this example. Define

$$Y_t^n := \int_{s_1\vee s_2}^t b(X_u)b_n'(X_u)D_{s_2}X_u du.$$
Clearly, $Y_t^n \rightarrow Y_t$ in $L^2(\Omega)$ for all $t\in [0,T]$. We only need to bound $\|D_{\cdot} Y_t^n\|_{L^2([0,T]\times \Omega)}$ uniformly in $n\geq 0$. Then
\begin{align*}
D_{s_3} Y_t^n &= \int_{s^\ast}^t b'(X_u)D_{s_3} X_u b_n'(X_u)D_{s_2}X_u du\\
&+ \int_{s^\ast}^t b(X_u)b_n''(X_u)D_{s_3}X_u D_{s_2}X_u du+\int_{s^\ast}^t b(X_u)b_n'(X_u)D_{s_3}D_{s_2}X_u du
\end{align*}
where $s^\ast := \max\{s_1,s_2,s_3\}$.

Then the critical term is
\begin{align*}
I_n := E\left[\int_0^T\left(\int_{s^\ast}^t b(X_u)b_n''(X_u)D_{s_3}X_u D_{s_2}X_u du\right)^2 ds_3 \right]
\end{align*}
Denote $\tilde{B}_{s,t}:= \exp\left\{\int_s^t b'(B_u^x)du \right\}$. Then, by Girsanov's theorem and Lemma \ref{epsilonbound} we have for a suitable $\varepsilon>0$
\begin{align*}
I_n&=\int_0^T E\left[\left(\int_{s^\ast}^t b(B_u^x)b_n''(B_u^x)\tilde{B}_{s_3,u} \tilde{B}_{s_2,u} du\right)^2 \ \mathcal{E}\left(\int_0^T b(B_u^x)dB_u\right)\right]ds_3\\
&\leq C_{\varepsilon} \int_0^T E\left[\left(\int_{s^\ast}^t b(B_u^x)b_n''(B_u^x)\tilde{B}_{s_3,u} \tilde{B}_{s_2,u} du\right)^{2\frac{1+\varepsilon}{\varepsilon}}\right]^{\frac{\varepsilon}{1+\varepsilon}} ds_3.
\end{align*}
Now we focus on the expectation. Choose $\varepsilon>0$ so that $p:= 2\frac{1+\varepsilon}{\varepsilon}$ is a natural number. Then since $|\tilde{B}_{s_3,{u_i}} \tilde{B}_{s_2,{u_i}}|\leq e^{(|t-s_3|+|t-s_2|)\|b'\|_{\infty}}\leq C<\infty$ and since $b$ and $b_n''$ are positive we have
\begin{align*}
\Bigg|E\bigg[&\int_{s^\ast}^t\cdots \int_{s^\ast}^t \prod_{i=1}^p b(B_{u_i}^x)b_n''(B_{u_i}^x)\tilde{B}_{s_3,{u_i}} \tilde{B}_{s_2,{u_i}} du_1 \cdots du_p\bigg]\Bigg| \leq \\
&\leq E\bigg[\int_{s^\ast}^t\cdots \int_{s^\ast}^t \prod_{i=1}^p \big| b(B_{u_i}^x)b_n''(B_{u_i}^x)\tilde{B}_{s_3,{u_i}} \tilde{B}_{s_2,{u_i}}\big| du_1 \cdots du_p\bigg]\\
&\leq C E\bigg[\int_{s^\ast}^t\cdots \int_{s^\ast}^t \prod_{i=1}^p \big| b(B_{u_i}^x)b_n''(B_{u_i}^x)\big| du_1 \cdots du_p\bigg]\\
&= C E\bigg[\int_{s^\ast}^t\cdots \int_{s^\ast}^t \prod_{i=1}^p b(B_{u_i}^x)b_n''(B_{u_i}^x) du_1 \cdots du_p\bigg].
\end{align*}
Then since $(u_1,\dots,u_p)\mapsto b(B_{u_1}^x)b_n''(B_{u_1}^x)\cdots b(B_{u_p}^x)b_n''(B_{u_p}^x)$ is symmetric we may write
\begin{align*}
E\bigg[\int_{s^\ast}^t\cdots \int_{s^\ast}^t \prod_{i=1}^p b(B_{u_i}^x)b_n''(B_{u_i}^x)& du_1 \cdots du_p\bigg]\\
&\leq p! E\bigg[\int_{s^\ast<u_1<\cdots <u_p<t} \prod_{i=1}^p b(B_{u_i}^x)b_n''(B_{u_i}^x) du_1 \cdots du_p\bigg]
\end{align*}
and the last may be bounded independently of $b_n''$ by using Proposition \ref{mainEstimate}. In fact,
$$\sup_{s^\ast \in [0,T]} \sup_{n\geq 0} E\bigg[\int_{s^\ast<u_1<\cdots <u_p<t} \prod_{i=1}^p b(B_{u_i}^x)b_n''(B_{u_i}^x) du_1 \cdots du_p\bigg] \leq C$$
for a finite constant $C$. So
$$\sup_{n\geq 0} I_n <\infty$$
being thus $Y_t \in \mathbb{D}^{1,2}(\Omega)$ for every $t\in [0,T]$.

In a summary, we have in (\ref{twoMall}) a sum of Malliavin differentiable terms except for the last one $-4D_{s_1} X_t \int_{s_1\vee s_2}^t b'(X_u)D_{s_2}X_u dB_u$ . In conclusion $D_{s_2}D_{s_1} X_t \notin \mathbb{D}^{1,2}(\Omega)$.
\end{exam}

Finally, we give an extension of Theorem \ref{mainthm} and Corollary \ref{maincor} to a class of non-degenerate $d-$dimensional It\^{o}-diffusions$.$
\begin{thm}
\label{generalsde}Consider the time-homogeneous $\mathbb{R}^{d}-$valued SDE%
\begin{equation}
dX_{t}=b(X_{t})dt+\sigma (X_{t})dB_{t},\,\,\text{ }X_{0}=x\in \mathbb{R}^{d},%
\text{ }\,\,0\leq t\leq T,  \label{SDE1}
\end{equation}%
where the coefficients $b:\mathbb{R}^{d}\longrightarrow \mathbb{R}^{d}$ and $%
\sigma :\mathbb{R}^{d}\longrightarrow \mathbb{R}^{d}\times $ $\mathbb{R}^{d}$%
are Borel measurable. Require that there exists a bijection $\Lambda :%
\mathbb{R}^{d}\longrightarrow \mathbb{R}^{d}$, which is twice continuously
differentiable. Let \mbox{$\Lambda _{x}:\mathbb{R}^{d}\longrightarrow L\left( 
\mathbb{R}^{d},\mathbb{R}^{d}\right) $} and $\Lambda _{xx}:\mathbb{R}%
^{d}\longrightarrow L\left( \mathbb{R}^{d}\times \mathbb{R}^{d},\mathbb{R}%
^{d}\right) $ be the corresponding derivatives of $\Lambda $ and assume that%
\begin{equation*}
\Lambda _{x}(y)\sigma (y)=id_{\mathbb{R}^{d}}\text{ for }y\text{ a.e.}
\end{equation*}%
as well as%
\begin{equation*}
\Lambda ^{-1}\text{ is Lipschitz continuous.}
\end{equation*}%
Suppose that the function $b_{\ast }:\mathbb{R}^{d}\longrightarrow \mathbb{R}%
^{d}$ given by 
\begin{align*}
b_{\ast }(x)&:=\Lambda _{x}\left( \Lambda ^{-1}\left( x\right) \right) 
\left[ b(\Lambda ^{-1}\left( x\right) )\right] \\
&\,\,\,+\frac{1}{2}\Lambda _{xx}\left( \Lambda ^{-1}\left( x\right) \right) \left[
\sum_{i=1}^{d}\sigma (\Lambda ^{-1}\left( x\right) )\left[ e_{i}\right]
,\sum_{i=1}^{d}\sigma (\Lambda ^{-1}\left( x\right) )\left[ e_{i}\right] %
\right]
\end{align*}%
satisfies condition (\ref{bcondition}), where $e_{i},$ $%
i=1,\ldots,d$, is a basis of $\mathbb{R}^{d}.$ Then the conclusions of Theorem \ref{mainthm} and Corollary \ref{maincor} also apply to $X_t$, $t\in [0,T]$ and its density.
\end{thm}

\begin{proof}
The proof can be directly obtained from It\^o's Lemma. See \cite{MBP10}.
\end{proof}


\section{A classical solution to the stochastic transport equation}

The Sobolev regularity of the solution shown in Theorem \ref{mainprop} with respect to the initial condition entitles us to construct a classical solution to the \emph{stochastic transport equation} when the drift is Lipschitz which to our knowledge is not proved.

The Stochastic Transport Equation is written in differential form
\begin{align}\label{STE}
\begin{cases}
d_t u(t,x) + \nabla u(t,x) \cdot b(t,x) dt + \sum_{i=1}^d  e_i \cdot \nabla u(t,x) \circ dB_t^{(i)} =0\\
u(0,x) = u_0(x),
\end{cases}
\end{align}
where $b:[0,T] \times \R^d \rightarrow \R^d$ is a given vector field and $u_0: \R^d \rightarrow \R$ a given initial data. The stochastic integration is understood in the Stratonovich sense.

\begin{defi}[Classical solution]\label{solutionDefinition}
Let $u_0$ and $b$ be given functions. We say that a stochastic process \mbox{$u \in L^{\infty}(\Omega \times [0,T] \times \mathbb{R}^d)$} is a classical solution to (\ref{STE}) if 

\begin{enumerate}

\item
There exists a measurable set $\tilde{\Omega} \subset \Omega$ with full measure such that for fixed $t \in [0,T]$ and $p \geq 1$, the mapping $x \mapsto u(\omega,t,x)$ is in $W^{2,p}_{loc}(\mathbb{R}^d)$ on $\tilde{\Omega}$;

\item
For fixed $x \in \mathbb{R}^d$ there are $(\mathcal{F}_t)$-adapted versions of $t \mapsto u(t,x)$ and $t \mapsto \nabla u(t,x)$;

\item 
The following integral equation is satisfied
\begin{equation} \label{integralSTE}
u(t,x) + \int_0^t b(s,x) \cdot \nabla u(s,x) ds + \sum_{i=1}^d \int_0^t e_i \cdot \nabla u(s,x) \circ dB_s^{(i)} = u_0(x)
\end{equation}
for a.e. $(\omega, x) \in \Omega \times \mathbb{R}^d$.
\end{enumerate}
\end{defi}

Notice that we are using the Stratonovich integral in our definition, but following the same idea as in  \cite{FGP10}, Lemma 13, we can recast (\ref{integralSTE}) in It\^{o}-form as 
\begin{equation} \label{integralSTEito}
u(t,x) + \int_0^t b(s,x) \cdot \nabla u(s,x) ds + \sum_{i=1}^d \int_0^t e_i \cdot \nabla u(s,x) dB_s^{(i)}  + \frac{1}{2} \int_0^t \Delta u(s,x) ds= u_0(x).
\end{equation}
We will use these formulations interchangeably.

Before we proceed further we will introduce the concept of stochastic flow associated to SDE (\ref{SDE}):

\begin{defi}[Stochastic flow of diffeomorphisms]\label{flowDefinition}
A function $\phi: [0,T]\times [0,T] \times \R\times \Omega \rightarrow \R$, $\phi_{s,t}(x,\omega)$ is said to be a stochastic flow of diffeomorphisms of the SDE (\ref{SDE}) if there exists a full-measure set $\tilde{\Omega}\in \mathcal{F}$ such that for any $\omega \in \tilde{\Omega}$ the following holds true:
\begin{itemize}
\item[(i)]  $\phi_{s,t}(x,\omega)$, $s,t\in [0,T]$, $x\in \R$ is a (global) strong solution to the SDE (\ref{SDE}).
\item[(ii)] $\phi_{s,t}(x,\omega)$ is continuous in $(s,t,x)\in [0,T] \times [0,T] \times \R$.
\item[(iii)] $\phi_{s,t}(\cdot, \omega) = \phi_{u,t}(\cdot, \omega) \circ \phi_{s,u}(\cdot,\omega)$ for any $s,u,t\in [0,T]$ .
\item[(iv)] $\phi_{s,s}(x,\omega)= x$ for all $x\in \R$ and $s\in [0,T]$.
\item[(v)] $\phi_{s,t}(\cdot,\omega):\R^d \rightarrow \R^d$ are diffeomorphisms (of class $C^k$) for all $s,t\in [0,T]$.
\end{itemize}
\end{defi}

For the rest of this section we will assume that $b$ satisfies condition (\ref{bcondition}) for $k=1$ which in particular means that $b$ is globally Lipschitz, uniformly in time.

To get a globally defined (i.e. on the entire $\mathbb{R}^d$) stochastic flow of diffeomorphisms of the SDE (\ref{SDE}) we notice that since $b$ is uniformly Lipschitz  there exists a unique solution to
$$
\phi_{s,t}(x,\omega) = x + \int_s^t b(r, \phi_{s,r}(x, \omega)) dr + B_t(\omega) - B_s(\omega)
$$
for \emph{all} $\omega \in \Omega$.

It is easy to check conditions (i) to (iv) in Definition \ref{flowDefinition} holds for all $\omega \in \Omega$.

Fix $p \geq 1$ and $N \in \mathbb{N}$ and invoke Theorem \ref{sobreg} to guarantee that there exists a measurable subset $\Omega_N \subset \Omega$ with full measure such that the local solution 
$$
\phi^N_{s,t}(x,\omega) = x + \int_s^t b(r, \phi^N_{s,r}(x, \omega)) dr + B_t(\omega) - B_s(\omega).
$$

satisfies $\phi^N_{s,t}(\cdot , \omega) \in W^{2,p} ( B(0,N))$ for all $\omega \in \Omega_N$ and $x \in B(0,N)$. By uniqueness we have that $\phi_{s,t} |_{B(0,N) \times \Omega_N} = \phi^N_{s,t}$. 

If we let $\tilde{\Omega} := \cap_{N=1}^{\infty} \Omega_N$, we get that $\mathbb{R}^d \ni x \mapsto \phi_{s,t}(x,\omega) \in \mathbb{R}^d$ is twice weakly differentiable for every $\omega \in \tilde{\Omega}$, and thus condition (v) in \ref{flowDefinition} is satisfied for $k=1$.

\bigskip

In \cite{FGP10} the authors study (\ref{STE}) under the considerably weaker condition (at least for $d > 1$), $b\in L_{loc}^1([0,T] \times \mathbb{R}^d; \mathbb{R}^d)$,  div$b \in L_{loc}^1([0,T] \times \mathbb{R}^d)$ and $u_0 \in L^{\infty}(\mathbb{R}^d)$. However, in this case, one is restricted to study analytically weak solutions in the sense that for every test function $\theta \in C^{\infty}_0(\mathbb{R}^d)$ one has
\begin{align} \label{weakSTE}
\int_{\R^d} u(t,x)\theta(x)dx &= \int_{\R^d} u_0(x) \theta(x)dx\\
\notag&+ \int_0^t \int_{\R^d} u(s,x)\left[b(t,x)\cdot \nabla \theta(x) + \textrm{div} b(t,x) \theta (x)\right]dxds\\
\notag &+\sum_{i=1}^d \int_0^t \left( \int_{\R^d} u(s,x)\cdot\frac{\partial}{\partial x_i}\theta(x) dx\right) \circ dB_s^{(i)} .
\end{align}

Moreover, the equation is uniquely solved by $u(t,x) = u_0 (\phi_t^{-1}(x))$. 

Although we consider more restrictive coefficients, we arrive at an analytically stronger solution:

\begin{thm}
Let $b$ satisfy condition (\ref{bcondition}) for $k=1$ and $u_0 \in C_b^2(\mathbb{R}^d)$. Then there exists a unique classical solution to the stochastic transport equation.

Moreover, the equation is explicitly solved by $u(t,x) = u_0(\phi_t^{-1}(x))$.

\end{thm}

\begin{proof}
By the above discussion we know that for every test function $\theta \in C_0^{\infty}(\mathbb{R}^d)$, the equation (\ref{weakSTE}) is satisfied $P$-a.s. by $u(t,x) = u_0(\phi_t^{-1}(x))$. We now choose $\tilde{\Omega}$ such that $x \mapsto \phi_t^{-1}(x)$ is in $W^{2,p}_{loc}(\mathbb{R}^d)$ on $\tilde{\Omega}$. Then we get that $u$ satisfies condition (i) and (ii) of Definition \ref{solutionDefinition}, and by integration by parts we have
\begin{align*}
\int_{\R^d} u(t,x)\theta(x)dx &= \int_{\R^d} u_0(x) \theta(x)dx\\
\notag&- \int_0^t \int_{\R^d} \nabla u(s,x) \cdot b(t,x) \theta (x)dxds\\
\notag &-\sum_{i=1}^d \int_0^t \left( \int_{\R^d} \nabla u(s,x)\theta(x) dx\right) \circ dB_s^{(i)} .
\end{align*}

or equivalently
\begin{align*}
\int_{\R^d} u(t,x)\theta(x)dx &= \int_{\R^d} u_0(x) \theta(x)dx\\
\notag&- \int_{\R^d}\int_0^t \nabla u(s,x) \cdot b(t,x) \theta (x)dsdx\\
\notag &-\int_{\R^d} \sum_{i=1}^d \int_0^t   \nabla u(s,x)t dB_s^{(i)} \theta(x) dx - \frac{1}{2} \int_{\mathbb{R}^d} \int_0^t \Delta u(s,x) ds \theta(x)dx.
\end{align*}

Since $\theta$ was arbitrary, this proves the claim.
\end{proof}

\textbf{Acknowledgement:} We would like to thank the referees for their valuable comments and suggestions.



\end{document}